\documentclass[12pt,psamsfonts]{article}
\usepackage{amsmath}
\usepackage{amssymb}
\usepackage[mathcal]{eucal}
\usepackage[all]{xy}
\usepackage{color}

\newtheorem{deff}{Definition}[section]

\newtheorem{lemma}[deff]{Lemma}

\newtheorem{theorem}[deff]{Theorem}
\newtheorem{coro}[deff]{Corollary}

\newtheorem{propo}[deff]{Proposition}

\newtheorem{em-example}[deff]{Example}

\newtheorem{em-def}[deff]{Definition}

\newtheorem{em-remark}[deff]{Remark}

\newtheorem{em-question}[deff]{Question}

\newtheorem{problem}[deff]{Problem}

\newenvironment{defi}{\begin{em-def} \em  }{ \end{em-def}}
\newenvironment{remark}{\begin{em-remark} \em }{\end{em-remark}}

\newenvironment{proof}{\noindent {\it Proof}.}{\QED}

\newcommand{\palabras}{\noindent{\it Keywords: }{\small Reflexive;
Precompact; Pseudocompact; Baire property; Convergent sequence}}

\newcommand\proved{\vbox{\hrule\hbox{\vrule\vrule
width 0pt height 4pt depth3.5pt\hskip7pt\vrule}\hrule}}
\newcommand\QED{\hfill \proved \medskip}

\def\ker{\mathop{\rm ker}}

\def\sm{\setminus}
\def\sub{\subseteq}

\def\om{\omega}

\DeclareMathOperator*{\supp}{supp}

\DeclareMathSymbol{\res}{\mathord}{AMSa}{"16}

\def\:{\nobreak \hskip .1111em\mathpunct {}\nonscript \mkern
   -\thinmuskip {:}\hskip .3333emplus.0555em\relax}

\catcode`\@=12
\def\T{{\mathbb T}}

\def\Z{{\mathbb Z}}
\def\N{{\mathbb N}}

\title{On convergent sequences in dual groups}
\author{M.V. Ferrer 
\thanks{This author was partially supported by the Generalitat Valenciana,
grant GV/2018/110}
\and  
S.~Hern\'andez
\thanks{This author was partially supported by the Spanish Ministerio de Econom\'{i}a y Competitividad,
grant MTM2016-77143-P (AEI/FEDER, EU)}
 \and
M.~Tkachenko
\thanks{The article was finished  during the visit of the third listed author to the
Universitat Jaume I, Spain, in June, 2019. He expresses his gratitude to the
hosts for financial support and kind attention.}
}

\date{October 10, 2019}

\begin{document}
\maketitle

\begin{abstract}
We provide some characterizations of precompact abelian groups $G$ whose
dual group $G_p^\wedge$ endowed with the pointwise convergence topology
on elements of $G$ contains a nontrivial convergent sequence. In the special
case of precompact abelian \emph{torsion} groups $G$, we characterize the
existence of a nontrivial convergent sequence in $G_p^\wedge$ by the following
property of $G$: \emph{No infinite quotient group of $G$ is countable.} Finally,
we present an example of a dense subgroup $G$ of the compact metrizable
group $\mathbb{Z}(2)^\omega$ such that $G$ is of the first category in itself,
has measure zero, but the dual group $G_p^\wedge$ does not contain
infinite compact subsets. This complements Theorem 1.6 in [J.E.~Hart and
K.~Kunen, Limits in function spaces and compact groups, \textit{Topol. Appl.}
\textbf{151} (2005), 157--168]. As a consequence, we obtain an example of
a precompact reflexive abelian group which is of the first Baire category.
\end{abstract}

\palabras
\medskip


\section{Introduction}

Our aim is to study the class of precompact topological abelian groups $G$ such that
the dual group $G_p^\wedge$ endowed with the pointwise convergence topology on
elements of $G$ contains a nontrivial convergent sequence. It is well known that a group
$G$ from this class cannot be compact (see \cite{glick}). A more general result follows from
\cite[Proposition~4.4]{HM}: If $G_p^\wedge$ contains infinite compact subsets then $G$
is not pseudocompact. An even more general fact is established in \cite{FT}: If $G$ has
the Baire property, then $G_p^\wedge$ does not contain nontrivial convergent sequences.
We conjecture that $G_p^\wedge$ does not contain infinite compact subsets in this case
(see Problem~\ref{Prob:1}). This conjecture has been proved in \cite[Theorem~3.3]{CDT}
for bounded torsion groups $G$.

Therefore, we have to consider only precompact groups of the first Baire category,
i.e.~the groups which can be covered by countably many nowhere dense subsets. In
Theorem~\ref{Pro:Mea} we present an example of a dense subgroup $G$ of the compact
metrizable group $\mathbb{Z}(2)^\omega$ such that $G$ is of the first category in itself,
has measure zero, but the dual group $G_p^\wedge$ does not contain nontrivial convergent
sequences. Furthermore, since $G_p^\wedge$ is countable, all compact subsets of
$G_p^\wedge$ are finite. This complements \cite[Theorem~1.6]{HK} by J.E.~Hart and
K.~Kunen. We conclude, therefore, that the property of a precompact abelian group
$G$ to be \lq{small\rq} (to be Haar nullset in the completion of $G$ or to be of the
first category in itself, or both) does not guarantee the existence of nontrivial convergent
sequences in $G_p^\wedge$.

In the special case when $G$ is an infinite topological subgroup of the circle group
$\mathbb{T}$, our study is intimately related to the so-called \emph{characterized}
subgroups of $\mathbb{T}$ (see the articles \cite{DG,DGT} and the references therein).
In this case, the dual group $G^\wedge$ is algebraically isomorphic to the group
$\mathbb{Z}$ of integers. Given a strictly increasing sequence $B=(n_k)_{k\in\omega}$
of positive integers, one defines $C_B^{\mathbb{T}}$ to be the set of all $x\in\mathbb{T}$
such that $x^{n_k}\to 1$ when $k\to\infty$. A subgroup $G$ of $\mathbb{T}$ is
called \emph{characterized} if $G=C_B^{\mathbb{T}}$, for some sequence
$B\subset\omega$. It is clear from the definition that the sequence $B$ converges
to the identity of the dual group $\mathbb{Z}=(C_B^\mathbb{T})_p^\wedge$ and that
$C_B^\mathbb{T}$ is the biggest subgroup of $\mathbb{T}$ for which $B$ converges.
It is known that $C_B^{\mathbb{T}}$ has measure zero in $\mathbb{T}$, for each
$B\subset\omega$ \cite[Lemma~3.10]{CTW}. This is one of very few general results
about characterized subgroups of $\mathbb{T}$, though the articles dedicated to
their study abound. It is also worth mentioning that the index $[\mathbb{T}:G]$ of
a characterized subgroup $G\leq\mathbb{T}$ is uncountable\,---\,otherwise the
circle group could be covered by a countable family of nullsets, which is impossible.
Needless to say, no internal description of characterized subgroups of $\mathbb{T}$
is available now.

Summing up, searching for nontrivial convergent sequences in the dual groups
of dense subgroups of compact connected groups presents considerable difficulties.
Here, we tackle this question and present some conditions guaranteeing that the
dual group $G_p^\wedge$ of a precompact abelian group $G$ contains a nontrivial
convergent sequence. Clearly, if $G$ is a \emph{countable} infinite precompact
abelian group, then the dual group $G_p^\wedge$ has a countable base and is not
discrete. Hence $G_p^\wedge$ contains nontrivial convergent sequences. We use
this simple observation in the proofs of several results here. We focus our attention on
the study of the duals of precompact abelian \emph{torsion} groups $G$. In this special
case, we characterize in Theorem~\ref{Prop:98} the existence of nontrivial convergent
sequences in $G_p^\wedge$ by the following property of $G$: \emph{No infinite quotient
group of $G$ is countable}. Finally, an example of a reflexive, precompact, abelian group
of the first Baire category is provided.

\subsection{Notation and terminology}
The identity element of a group $G$ is denoted by $e_G$ or simply $e$ if no confusion is
possible. A \emph{character} of a group $G$ is an arbitrary homomorphism of $G$ to the
circle group $\mathbb{T}$. The latter group is identified with the multiplicative subgroup
of the complex numbers $z$ with $|z|=1$. If $G$ is a topological group, a \emph{continuous
character}  of $G$ is a continuous homomorphism of $G$ to $\mathbb{T}$ provided the
latter group carries its usual compact topology inherited from the complex plane $\mathbb{C}$.
We also set $V_1=\{z\in\mathbb{T}: -\pi/2<Arg(z)<\pi/2\}$.

For a given topological group $G$, its \emph{dual group} is denoted by $G^\wedge$. The
dual group consists of the continuous characters of $G$ with the pointwise multiplication,
$(\chi_1\cdot\chi_2)(x)=\chi_1(x)\cdot \chi_2(x)$ for each $x\in G$. The identity element
of $G^\wedge$ is the constant homomorphism $e\colon G\to\mathbb{T}$, $e(x)=1$ for
each $x\in G$. Here $1$ is the identity element of $\mathbb{T}$.

In this article, the dual group $G^\wedge$ will always be endowed with the pointwise
convergence topology whose local base at the identity $e$ is formed by the sets
$$
O(x_1,\ldots,x_n) = \{\chi\in G^\wedge: \chi(x_k)\in V_1\, \mbox{ for each } k=1,\ldots,n\},
$$
where $n$ is a positive integer and $x_1,\ldots,x_n\in G$. The group $G^\wedge$ equipped with
the pointwise convergence topology on $G$ is denoted by $(G^\wedge ,t_p(G))$ or by $G_p^\wedge$, for short,
when there is no possible confusion.

There exist infinite Hausdorff topological abelian groups $G$ with the tri\-vial dual
group $G^\wedge$ (see \cite{Pro,AHK}). If, however, the group $G$ is \emph{precompact}
(equivalently, is topologically isomorphic to a dense subgroup of a compact topological
group), then the continuous characters of $G$ separate points of $G$. Furthermore,
in this case, the topology of $G$ coincides with the topology of pointwise convergence
on elements of the dual group $G^\wedge$. This follows from a theorem of Comfort and
Ross (see \cite[Theorem~1.2]{CR}).

If $N$ is a subgroup of a topological group $G$, we denote by $N^\bot$ the
\emph{annihilator} of $N$, i.e.~the subgroup of $G^\wedge$ which consists
of all characters $\chi\in G^\wedge$ satisfying $\chi(N)=\{1\}$. The group $N^\bot$
is always closed in $G_p^\wedge$. Similarly, if $M$ is a subgroup of $G^\wedge$,
$M^\bot$ denotes the annihilator of $M$ consisting of all elements $x\in G$
satisfying $\chi(x)=1$ for each $\chi\in M$. It is clear that $M^\bot$ is the
intersection of the kernels of the characters of $M$. Again, $M^\bot$ is a
closed subgroup of $G$.

Let $f\colon G\to H$ be a continuous homomorphism of topological groups. We
define the \emph{dual} homomorphism $f^\wedge\colon H_p^\wedge\to G_p^\wedge$
by letting $f^\wedge(\chi)=\chi\circ{f}$ for each $\chi\in H_p^\wedge$. It is easy
to see that $f^\wedge$ is continuous. It is also known that $f^\wedge$ is injective
if $f(G)=H$ and that $f^\wedge$ is onto if $f$ is a topological monomorphism
and $H$ (hence, $G$) is precompact and abelian \cite{DPS,Mor}.

\section{Duals of precompact torsion groups}\label{Sec:Res}

The following dichotomy for homogeneous spaces is a kind of the topological folklore
(see \cite[Theorem~2.3]{LM}).

\begin{lemma}\label{Le:Bai}
If a homogeneous space $X$ is not Baire, then $X$ is of the
first category in itself.
\end{lemma}

%

In what follows we identify the two-element cyclic group $\Z(2)$ with
the subgroup $\{1,-1\}$ of the multiplicative circle group $\T$. We
start with a simple and well known lemma that will be applied in the proof of
Theorem~\ref{Pro:Mea}. We include a short proof of it here for completeness sake.

\begin{lemma}\label{Le:S}
Let $\chi\colon \Z(2)^A\to \T$ be a homomorphism, where $|A|<\omega$.
Then there exists a set $B\subset A$ such that $\chi(x)=\prod_{i\in B} x(i)$,
for each $x\in \Z(2)^A$.
\end{lemma}

\begin{proof}
If $H=\prod_{i\in A} H_i$ is a product of finitely many Abelian groups and
$\chi\colon H\to \T$ is a homomorphism, then there exist homomorphisms
$\chi_i\colon H_i\to\T$, for $i\in A$, such that $\chi(x)=\prod_{i\in A}\chi_i(x_i)$
for each $x\in\Z(2)^A$. Since every homomorphism of $\Z(2)$ to $\T$ is
either trivial or a monomorphism, the conclusion of the lemma is now immediate.
\end{proof}

A subset $B$ of a Tychonoff space $X$ is called \emph{bounded} in $X$
if the image $f(B)$ is a bounded subset of the real line, for every continuous
real-valued function $f$ on $X$ (see \cite{HST} or \cite[Section~6.10]{AT}).
It is clear that all compact subsets of $X$ are bounded. Conversely, in a
Diedonn\'e complete space $X$, the closure of every bounded subset
is compact \cite[Proposition~6.10.1\,c)]{AT}.

In the following theorem, we denote by $\mu$ the Haar measure of the
compact group $\Z(2)^\omega$ (see \cite[Chapter~10]{Hal}).

\begin{theorem}\label{Pro:Mea}
There exists an infinite first category subgroup $G$ of the compact group
$\Z(2)^\omega$ such that $\mu(G)=0$ and every bounded subset of the dual
group $G^\wedge_p$ is finite. In particular, $G^\wedge$ does not contain
non-trivial convergent sequences.
\end{theorem}

\begin{proof}
Let call a subset $A$ of $\omega$ \textit{thin} if
\[
\lim_{k\to\infty} \frac{|k\cap A|}{k}=0,
\]
where each positive integer $k\in\omega$ is identified with the set
$\{0,\ldots,k-1\}$. For an element $x\in \Z(2)^\omega$, let
\[
\supp(x)=\{n\in\omega: x(n)=1\}.
\]
Denote by $G$ the set of all $x\in \Z(2)^\omega$ such that $\supp(x)$
is a thin subset of $\omega$. It is clear that $G$ is a dense
subgroup of $\Z(2)^\omega$ and that $G$ is precompact.

First, we claim that $G$ is of the first category in $\Z(2)^\omega$.
Indeed, for positive integers $m$ and $N$, let
\[
O_{m,N}=\{x\in \Z(2)^\omega: \exists\, k\geq m\, \mbox{ such that }\,
\frac{|\supp(x)\cap k|}{k}\geq 1/N\}.
\]
It is easy to see that the sets $O_{m,N}$ are open and dense in
$\Z(2)^\omega$ whenever $N\geq 2$, while our definition of $G$
implies that $G$ is disjoint from the set $\bigcap_{k,N\geq 2}
O_{m,N}$. Hence $G$ is contained in the first category set
$\bigcup_{m,N\geq 2} F_{m,N}$, where $F_{m,N}=\Z(2)^\omega\sm
O_{m,N}$. This proves our claim.

Let us show that $\mu(G)=0$, where $\mu$ is the Haar measure
on the compact group $\Z(2)^\omega$. It is easy to see that $G$ is
measurable. Indeed, every set $O_{m,N}$ is open in the compact
second countable group $\Z(2)^\omega$ and, hence, is measurable.
Hence the sets $F_{m,N}=\Z(2)^\omega\sm O_{m,N}$ are also
measurable. The definition of the group $G$ implies that
$$
G=\bigcap_{N=1}^\infty \bigcup_{m=0}^\infty F_{m,N},
$$
whence it follows that $G$ is measurable. Since the index of $G$
in $\Z(2)^\omega$ is infinite, no finite number of cosets of $G$ in
$\Z(2)^\omega$ covers the group $\Z(2)^\omega$. Hence, by the
additivity of $\mu$, we conclude that $\mu(G)=0$.

Now we verify that the dual group $G^\wedge_p$ does not contain
non-trivial sequences converging to the identity element of $G^\wedge_p$.
Consider a sequence $\{\chi_n: n\in\om\}\sub G^\wedge_p$. We can
assume without loss of generality that $\chi_n\neq \chi_m$ if
$n\neq m$. Since $G$ is dense in $\Z(2)^\omega$, every $\chi_n$
extends to a continuous character on $\Z(2)^\omega$. We denote this
extension by $\varphi_n$, for each $n\in\omega$. According to \cite{Kap},
every $\varphi_n$ depends on finitely many coordinates, i.e.~one can find
a finite set $A_n\subset\omega$ and a character $\psi_n$ on $\Z(2)^{A_n}$
such that $\varphi_n=\psi_n\circ\pi_{A_n}$, where $\pi_{A_n}\colon
\Z(2)^\omega \to \Z(2)^{A_n}$ is the projection. Clearly, we can
assume that $A_n$ is a minimal (by inclusion) set with this
property. It then follows from Lemma~\ref{Le:S} that
$\psi_n(x)=\prod_{i\in A_n} x(i)$, for each $x\in \Z_2^{A_n}$.

For every finite $A\subset\omega$, there exist only finitely many
homomorphisms of $\Z(2)^A$ to $\T$. Hence, choosing a subsequence
of $\{\chi_n: n\in\omega\}$, if necessary, we can assume that the sets
$A_n$ are non-empty and that each $A_n$ is not covered by the sets
$A_m$ with $m<n$. For every $n\in\om$, let $k_n$ be the biggest
element of the complement $A_n\sm\bigcup_{m<n} A_m$. Choosing
a subsequence of $\{\chi_n: n\in\om\}$ once again, we can assume
additionally that the set $K=\{k_n: n\in\om\}$ is thin.

Let us define an element $x^*\in G$ as follows. Choose $x_0\in
\Z(2)^{A_0}$ such that $\psi_0(x_0)=-1$ and $x_0(k)=1$ for each $k\in
A_0$ distinct from $k_0$. Assume that for some $n\in\om$, we have
defined elements $x_i\in\Z(2)^{A_i}$ with $i\leq n$ satisfying the
following conditions:
\begin{enumerate}
\item[(a)] $x_i(k)=x_j(k)$ whenever $j<i$ and $k\in A_i\cap A_j$;
\item[(b)] if $x_i(k)=-1$, then $k=k_j$ for some $j\leq i$;
\item[(c)] $\psi_i(x_i)=-1$.
\end{enumerate}
Let $y_1$ and $y_2$ be elements of $\Z(2)^{A_{n+1}}$ such that
$y_1(k)=y_2(k)=x_i(k)$ whenever $k\in A_{n+1}\cap A_i$ for some
$i\leq n$, $y_1(k)=y_2(k)=1$ for each $k\in A_{n+1}\sm
\bigcup_{i\leq n} A_i$ distinct from $k_{n+1}$, and
$y_1(k_{n+1})\neq y_2(k_{n+1})$. It follows from the choice of
$A_{n+1}$ and $k_{n+1}\in A_{n+1}$ that $\psi_{n+1}(y_1)\neq
\psi_{n+1}(y_2)$. Hence $\psi_{n+1}(y_i)=-1$ for some $i\in\{1,2\}$.
We put $x_{n+1}=y_i$. Then the elements $x_0,x_1,\ldots,x_{n+1}$
satisfy conditions (a)--(c).

Denote by $x^*$ an element of $\Z(2)^\omega$ such that
$\pi_{A_n}(x^*)=x_n$ for each $n\in\om$ (we apply (a) here) and
$x^*(k)=1$ if $k\in\om\sm\bigcup_{n\in\om} A_n$. Since the set
$K=\{k_n: n\in\om\}$ is thin, (b) implies that $\supp(x^*)\sub K$
and hence $x^*\in G$. It follows from (c) and our definition of $x^*$
that $\chi_n(x^*)= \psi_n(\pi_{A_n}(x^*))=\psi_n(x_n)=-1$, for
each $n\in\om$. Therefore, the sequence $\{\chi_n: n\in\omega\}$
does not converge to the identity in $G^\wedge_p$. This implies
that $G$ does not contain non-trivial convergent sequences at all.

Since $G$ is second countable and precompact, the dual group
$G_p^\wedge$ is countable. Suppose that $B$ is an infinite bounded
subset of $G_p^\wedge$. Then so is $\overline{B}$, the closure of $B$
in $G_p^\wedge$. Also, it follows from $|G_p^\wedge|=\omega$ that
the space $G_p^\wedge$ is Lindel\"of and, hence, Dieudonn\'e complete.
Therefore, $\overline{B}$ is compact. But every countably infinite compact
space contains non-trivial convergent sequences, which is a contradiction.
\end{proof}

\begin{remark}
We have shown in Theorem~\ref{Pro:Mea} that all bounded subsets of
$G_p^\wedge$ are finite. This is equivalent to saying that every infinite set
$X\sub G_p^\wedge$ contains an infinite subset $D$ such that $D$ is closed
and discrete in $G_p^\wedge$ and $C$-embedded in $G_p^\wedge$. The
equivalence follows immediately from the fact that the space $G_p^\wedge$ is
countable and regular, hence, normal.
\end{remark}

\begin{propo}\label{Prop:6}
Let $G$ be a precompact topological abelian group. The following conditions
are equivalent:
\begin{enumerate}
\item[{\rm (1)}] $G$ has a countably infinite Hausdorff quotient;
\item[{\rm (2)}] $G$ has a countably infinite Hausdorff homomorphic image;
\item[{\rm (3)}] $G_p^\wedge$ contains an infinite metrizable subgroup.
\end{enumerate}
\end{propo}

\begin{proof}
(2)\,$\Rightarrow$\,(1). Let $\varphi\colon G\longrightarrow H$ be a continuous
homomorphism onto a countably infinite Hausdorff topological group $H$. Then
$N=\ker\varphi$ is a closed subgroup of $G$. Let $\pi\colon G\to G/N$ be the quotient
homomorphism. Clearly, there is a continuous homomorphism $\varphi_N\colon G/N
\longrightarrow H$ such that $\varphi=\varphi_N\circ \pi$. It is clear from the definition
that $\varphi_N$ is one-to-one, which implies that $G/N$ is countable and Hausdorff.

(1)\,$\Rightarrow$\,(3). Let $\pi\colon G\longrightarrow H$ be a quotient homomorphism
onto a countably infinite Hausdorff topological group $H$. Clearly $H$ is precompact.
Hence $H_p^\wedge$ is an infinite metrizable group and $\pi^\wedge\colon H_p^\wedge\to
G_p^\wedge$ is a monomorphism, which is easily seen to be a topological embedding in
$G_p^\wedge$. This shows that $\pi^\wedge(H_p^\wedge)$ is an infinite metrizable subgroup
of $G_p^\wedge$.

(3)\,$\Rightarrow$\,(2). Let $\Gamma$ be an infinite metrizable subgroup of $G_p^\wedge$.
Then so is $\overline{\Gamma}$, the closure of $\Gamma$ in $G_p^\wedge$ (see
\cite[Proposition~1.4.16]{AT}). Hence we can assume that $\Gamma$ is closed in $G_p^\wedge$.
Let $i\colon \Gamma\to G_p^\wedge$ be the identity embedding. Then the dual homomorphism
$i^\wedge\colon (G_p^\wedge)_p^\wedge\to \Gamma_p^\wedge$ is continuous and surjective,
while the kernel of $i^\wedge$ is $\Gamma^\bot$. According to \cite{RT}, the canonical
evaluation mapping of $G$ to $(G_p^\wedge)_p^\wedge$ is a topological isomorphism,
so we can identify the groups $G$ and $(G_p^\wedge)_p^\wedge$. We have thus shown
that the abstract groups $G/\Gamma^\bot$ and $\Gamma^\wedge$ are isomorphic. Since
the latter group is countable and infinite as the dual of an infinite metrizable precompact
group and $\Gamma^\bot$ is closed in $G$, we conclude that $G/\Gamma^\bot$ is a
countably infinite Hausdorff quotient of $G$. This completes the proof.
\end{proof}

As an obvious corollary we obtain:

\begin{coro}
Let $G$ be a precompact abelian group. If $G_p^\wedge$ does not contain non-trivial
convergent sequences, then no infinite Hausdorff continuous homomorphic image of
$G$ is countable.
\end{coro}

Let us turn to torsion abelian groups. In this case, we characterize the existence of
nontrivial convergent sequences in the dual group $G_p^\wedge$ by any of the
equivalent conditions on $G$ given in Proposition~\ref{Prop:6}.

In the proof of Theorem~\ref{Prop:98}
we denote by $D^{(\omega)}$ the subgroup of $D^\omega$ which consists of all
elements which differ from the identity of $D$ on at most finitely many coordinates.

\begin{theorem}\label{Prop:98}
For a precompact torsion abelian group $G$, the following conditions are equivalent:
\begin{enumerate}
\item[{\rm (1)}] $G$ has a countably infinite Hausdorff quotient;
\item[{\rm (2)}]  $G$ has a countably infinite Hausdorff homomorphic image;
\item[{\rm (3)}]  $G_p^\wedge$ contains an infinite metrizable subgroup;
\item[{\rm (4)}]  $G_p^\wedge$ contains a non-trivial convergent sequence.
\end{enumerate}
\end{theorem}

\begin{proof}
The equivalence of (1), (2) and (3) was established in Proposition~\ref{Prop:6}.
The implication (3)\,$\Rightarrow$\,(4) is evident since an infinite precompact
group is non-discrete. Therefore we only need to verify that (4)\,$\Rightarrow$\,(2).

Let  $\{\chi_n: n\in\omega\}$ be a non-trivial sequence converging to the identity
element of $G_p^\wedge$. We can assume that $\chi_n\neq\chi_m$ if $n\neq m$.
Let $f$ be the diagonal product of the family $\{\chi_n: n\in\omega\}$. Clearly
$f$ is a continuous homomorphism of $G$ to the product group $\mathbb{T}^\omega$.
Take an arbitrary element $x\in G$ with $x\neq e$ and let $m$ be the order of $x$.
Then $m$ is finite since $G$ is a torsion group. Each value $\chi_n(x)$ belongs to
the cyclic subgroup of $\mathbb{T}$ generated by an element $b\in\mathbb{T}$
of order $m$. Since the sequence $\{\chi_n(x): n\in\omega \}$ converges to
$1\in \mathbb T$ for every $x\in G$, the latter implies that there is an integer
$n(x)\geq 0$ such that $\chi_n(x)=1$ for all $n\geq n(x)$. In other words,
$f(x)\in D^{(\omega)}$ for all $x\in G$, where $D$ is the torsion subgroup
of $\mathbb{T}$. This implies that the image $f(G)$ is countable.

We claim that $f(G)$ is infinite. Indeed, it follows from our definition of $f$ that
for every $n\in\omega$, there exists a continuous character $\mu_n$ of $f(G)$
satisfying $\chi_n=\mu_n\circ{f}$. Notice that the characters $f^\wedge(\mu_n)=
\mu_n\circ{f}=\chi_n$ of the group $G$ are pairwise distinct, so the set
$\{\mu_n: n\in\omega\}\subset f(G)_p^\wedge$ is infinite. Hence the group
$f(G)$ is infinite as well.
\end{proof}

\section{Compact subsets of dual groups}\label{GenCase}

In this section we are concerned with the topological groups that are not necessarily torsion.
First, we present a machinery for finding arbitrary compact subsets in the duals of precompact
abelian groups. In order to do so, a basic ingredient will be the group $C_p(K,\T)$ of all
continuous functions of a compact space $K$ into the torus $\T$, equipped with the pointwise
convergence topology.

Since $C_p(K,\T)$ is dense in the compact group $\T^K$, it follows that the dual
group of $C_p(K,\T)$ is algebraically isomorphic to $A(K)$. Taking into account that, throughout
this article, the dual group is always equipped with the pointwise convergence topology, we have
that the topological dual of $C_p(K,\T)$ is the group $A(K)$ equipped with the topology of pointwise
convergence on elements $C(K,\T)$. This topology coincides with the Bohr topology of the free abelian
group $A(K)$ (cf. \cite{GalHer:FM}). Thus we have:

\begin{propo}\label{Dual_C_p2}
If $K$ is a compact space, then the dual group of $C_p(K,\T)$ endowed with the
pointwise convergence topology is topologically isomorphic to the free abelian topological
group $A(K)$ on $K$ equipped with its Bohr topology.
\end{propo}

We remark that the free abelian group $A(K)$ \emph{respects compactness},
which means that every Bohr-compact subset of $A(K)$ is also compact in $A(K)$
(see \cite[Corollary~4.20]{GalHer:FM}).

\begin{defi}\label{Def_3.1}
Given a compact space $K$, we say that a subgroup $G$ of $C_p(K,\T)$ is \emph{separating}
if for any two distinct elements $x,y\in K$, there is $g\in G$ such that $g(x)\not=g(y)$.
\end{defi}

Let $G$ be a topological abelian group. As a consequence of Proposition~\ref{Dual_C_p2}
we obtain the following result:

\begin{propo}\label{Pr_3.1}
The dual $G_p^\wedge$ of a precompact abelian group $G$ contains a copy of a compact
space $K$ if and only if there is a continuous homomorphism of $G$ onto a separating
subgroup of $C_p(K,\T)$.
\end{propo}

\begin{proof}
\emph{Sufficiency}: Let $\varphi\colon G\longrightarrow H$ be a continuous homomorphism
onto a separating subgroup $H\subseteq C_p(K,\T)$. We have that $A(K)$ is algebraically
isomorphic to $C_p(K,\T)^\wedge$. Consider the continuous dual homomorphism
$$
\varphi^\wedge\colon (A(K),{t_p(C(K,\T)}))\longrightarrow G_p^\wedge.
$$
Since $\varphi(G)$ is separating, it follows that $\varphi^\wedge\res{K}$ is one-to-one.

\emph{Necessity}: Assume that $K\subseteq G_p^\wedge$. Then there is a continuous
homomorphism $\phi \colon (A(K), {t_p(C(K,\T))})\to G_p^\wedge$, which
is the identity map on $K$. Therefore, $\phi^\wedge(G)$ is a separating subgroup of
$C_p(K,\T)$, where $\phi^\wedge\colon G\to C_p(K,\T)$ is the dual homomorphism.
\end{proof}

\begin{coro}\label{Cor_3.1}
There are precompact abelian groups whose (precompact) dual groups contain infinite
compact subsets but no non-trivial convergent sequences.
\end{coro}

\begin{proof}
By Proposition~\ref{Dual_C_p2}, the dual group of $C_p(\beta \omega,\T)$ is
topologically isomorphic to $A(\beta\omega)$ equipped with the topology of pointwise
convergence on $C(\beta \omega,\T)$ and the latter group contains a copy of
$\beta\omega$. Since $A(\beta\omega)$ does not contain nontrivial
convergent sequences (see \cite[Theorem~3.4]{EOY} or \cite[Proposition~2.4]{Tk13}),
neither does $(A(\beta\omega), t_p(C(\beta \omega,\T))$. This completes the proof.
\end{proof}

\begin{remark}
An infinite compact space $K$ is \emph{Efimov} if $K$ does not contain either
nontrivial convergent sequences or $\beta\omega$. Proposition~\ref{Dual_C_p2}
implies that  the existence of precompact abelian groups whose dual groups
contain a compact Efimov subspace is consistent with $ZFC$ (see \cite{Fedo}).
\end{remark}

We can now characterize the groups whose dual groups do contain non-trivial convergent
sequences. We denote by $c_0(\T)$ the set of all sequences in $\T$ converging to $1$.
Each such a sequence is considered as an element of the compact group $\T^\N$.
Hence $c_0(\T)$ can be identified with a subgroup of $\T^\N$ which is equipped with
the subspace topology.

\begin{theorem}
Let $G$ be a precompact abelian group. Then the dual group $G_p^\wedge$ contains
non-trivial convergent sequences if and only if there is a continuous homomorphism of
$G$ onto a separating subgroup of $c_0(\T)$.
\end{theorem}

\begin{proof}
Let $K:=\{1/n : n\in\N\}\cup\{0\}$ be a nontrivial convergent sequence. It is easy to see
that $C_p(K,\T)$ is topologically isomorphic to $c_0(\T)$ equipped with the topology
inherited from $\T^\N$. Then it suffices to apply
Proposition~\ref{Pr_3.1}.
\end{proof}

The group $c_0(\T)$ equipped with the sup-metric is complete and, hence, is a Polish
topological group. Therefore, there exist precompact abelian groups admitting a finer
non-discrete Polish topological group topology and whose dual groups contain non-trivial
convergent sequences.\smallskip

In the fall of 1990's several specialists in the Pontryagin duality theory were discussing
the problem whether every precompact reflexive abelian group was compact.
Counterexamples to this conjecture appeared in \cite{ACDT} and \cite{GM}, where the
authors established the existence of pseudocompact noncompact reflexive abelian groups.
Clearly every pseudocompact space has the Baire property. Afterward, a method of
constructing precompact non-pseudocompact reflexive abelian groups was presented
in \cite{BT}. All the groups in \cite{BT}  with this combination of properties had the Baire
property. These facts suggest the new conjecture that all precompact reflexive abelian
groups have the Baire property. In Theorem~\ref{Th:FC} we show that this conjecture is
false by presenting an example of a precompact reflexive abelian group which is of the
first category. Our arguments require two auxiliary results (for the first of them,
see \cite[Proposition~4.4]{HM}).

\begin{lemma}\label{Le:Pse}
If $G$ is an infinite pseudocompact abelian group, then all compact subsets of
the dual group $G_p^\wedge$ are finite.
\end{lemma}

\begin{lemma}\label{Le:Quo}
Let $H$ be a closed pseudocompact subgroup of a topological abelian group $G$.
If the dual group $(G/H)_p^\wedge$ does not contain infinite compact subsets (nontrivial
convergent sequences), then neither does $G_p^\wedge$.
\end{lemma}

\begin{proof}
Assume that all compact subsets of $(G/H)_p^\wedge$ are finite. Let $\pi\colon G\to G/H$
be the quotient homomorphism and $\pi^\wedge\colon (G/H)_p^\wedge\to G_p^\wedge$
be the dual homomorphism. Let also $r\colon G_p^\wedge\to H_p^\wedge$ be the restriction
mapping, $r(\chi)=\chi\res{H}$ for each $\chi\in G_p^\wedge$. Clearly $r$ is a continuous
homomorphism.

Assume for a contradiction that $K$ is an infinite compact subset of $G_p^\wedge$.
Then $r(K)$ is a compact subset of $H_p^\wedge$. Since $H$ is pseudocompact,
$r(K)$ is finite by Lemma~\ref{Le:Pse}. Hence the compact set $L=K\cap r^{-1}(\mu)\sub
G_p^\wedge$ is infinite for some $\mu\in r(K)$. Take an arbitrary element $\chi_0\in L$
and put $L^\ast= L\cdot\chi_0^{-1}$. Then $L^\ast$ is an infinite compact subset of
$G_p^\wedge$ and the restriction to $H$ of every element $\nu\in L^\ast$ is the
identity of $H_p^\wedge$. Hence for every $\nu\in L^\ast$ there exists a character
$\psi\in (G/H)_p^\wedge$ satisfying $\nu=\psi\circ\pi=\pi^\wedge(\psi)$. We
conclude, therefore, that $L^\ast\sub \pi^\wedge\big((G/H)_p^\wedge\big)$.
Since $\pi^\wedge$ is a topological monomorphism, the latter inclusion implies
that $(G/H)_p^\wedge$ contains an infinite compact subset. This contradicts
the lemma assumptions. We have thus proved that all compact subsets
of $G_p^\wedge$ are finite.

The argument in the case of convergent sequences is almost the same.
\end{proof}

A subgroup $H$ of a topological abelian group $G$ is said to be \emph{$h$-embedded}
in $G$ if \emph{every} homomorphism $h\colon H\to\T$ extends to a \emph{continuous}
homomorphism of $G$ to $\T$ (see  \cite{Tk88} or \cite[p.~290]{ACDT}). It is clear from
the definition that every homomorphism $h\colon H\to\T$ of an $h$-embedded subgroup
$H$ of $G$ is continuous.

Let $G$ be a topological abelian group. We recall that the dual group of $G$ endowed
with the compact-open topology is denoted by $G^\wedge$. If all compact subsets of
$G$ are finite, then $G^\wedge=G_p^\wedge$.

\begin{theorem}\label{Th:FC}
There exists a first category, precompact, reflexive abelian group.
\end{theorem}

\begin{proof}
Since the compact group $\Z(2)^\omega$ is second countable, one can apply
\cite[Theorem~4.2]{BT} to find a pseudocompact abelian group $S$ and a closed
pseudocompact subgroup $P$ of $S$ such that the quotient group $S/P$ is
topologically isomorphic to $\Z(2)^\omega$ and all countable subgroups of
$S$ are $h$-embedded in $S$. In what follows we identify the groups $S/P$
and $\Z(2)^\omega$. Let $G$ be the subgroup of $\Z(2)^\omega$ considered in
Example~\ref{Pro:Mea}. Denote by $\pi$ the quotient homomorphism of $S$ onto
$S/P$. Clearly, $H=\pi^{-1}(G)$ is a precompact subgroup of $S$, $H/P\cong G$,
and $\varphi=\pi\res{H}$ is an open continuous homomorphism of $H$ onto $G$.
As $G$ is of the first category, so is $H$. It remains to verify that the group
$H$ is reflexive.

First, since $H$ is a subgroup of $S$, all countable subgroups of $H$ are
$h$-embedded. Applying \cite[Proposition~2.1]{ACDT} we see that all compact
subsets of $H$ are finite. In particular, $H^\wedge=H_p^\wedge$. Also, it follows
from Theorem~\ref{Pro:Mea} that all compact subsets of $(H/P)_p^\wedge\cong
G_p^\wedge$ are finite. Since $P$ is pseudocompact, all compact subsets of
$P_p^\wedge$ are finite by Lemma~\ref{Le:Pse}. Therefore, Lemma~\ref{Le:Quo}
implies that all compact subsets of $H_p^\wedge$ are finite as well. We conclude,
therefore, that
$$
(H^\wedge)^\wedge=(H_p^\wedge)^\wedge=(H_p^\wedge)_p^\wedge.
$$
Also, the canonical evaluation mapping of $H$ to $(H_p^\wedge)_p^\wedge$
is a topological isomorphism of $H$ onto $(H_p^\wedge)_p^\wedge$ \cite{RT}.
Hence the group $H$ is reflexive.
\end{proof}

We finish with the following problem mentioned in the introduction:

\begin{problem}\label{Prob:1}
Let $G$ be a precompact topological abelian group with the Baire property.
Is it true that all compact subsets of $G_p^\wedge$ are finite?
\end{problem}


M.V.~Ferrer\\
{\small\em Dept. de Matem\'{a}ticas}, {\small\em Universitat Jaume I},
{\small\em Campus Riu Sec s/n, 12071 Castell\'on, Spain.}\\
{\small\em e-mail: mferrer@mat.uji.es}
\medskip

\noindent S.~Hern\'andez\\
{\small\em Dept. de Matem\'{a}ticas}, {\small\em Universitat Jaume I},
{\small\em Campus Riu Sec s/n, 12071 Castell\'on, Spain.}\\
{\small\em e-mail: hernande@mat.uji.es}
\medskip

\noindent M.~Tkachenko\\
{\small\em Dept.~de~Matem\'aticas}, {\small\em Universidad Aut\'onoma Metropolitana},
{\small\em Av. San Rafael Atlixco $\#\,186$,
Col.~Vicentina, Iztapalapa, C.P.~09340, Ciudad de M\'exico, Mexico.}\\
{\small\em e-mail: mich@xanum.uam.mx}\\

\end{document}